\documentclass[12pt]{article} 

\usepackage[utf8]{inputenc}
\usepackage[T1]{fontenc}
\usepackage{lmodern}
\usepackage{xspace} 
\usepackage[normalem]{ulem}
\usepackage{bbm} 

\usepackage[english]{babel}
\usepackage{verbatim}

\usepackage{amsmath}
\allowdisplaybreaks
\usepackage{amssymb}
\usepackage{amscd}
\usepackage{amsfonts} 
\usepackage[normalem]{ulem} 
\usepackage{amsthm}

\usepackage{graphicx} 
\usepackage{tabu}
\usepackage{epic,eepic}
\usepackage{color}   

\usepackage[usenames,dvipsnames]{pstricks}
\usepackage{epsfig}
\usepackage{pst-grad} 
\usepackage{pst-plot} 
\usepackage[space]{grffile} 
\usepackage{etoolbox} 

\usepackage{hyperref}   
\usepackage{url}

\usepackage{enumerate}

\newtheorem{thm}{Theorem}
\newtheorem{lem}[thm]{Lemma}
\newtheorem{cor}[thm]{Corollary}
\newtheorem{question}[thm]{Question}

\newcommand{\x}{\,\,}

\newcommand{\ignore}[1]{{}}

\newcommand{\vol}{\mathrm{vol}}

\newcommand{\calP}{\mathcal{P}}
\newcommand{\calQ}{\mathcal{Q}}

\newcommand{\rarr}{\rightarrow}

\newcommand{\floor}[1]{\lfloor{#1}\rfloor}
\newcommand{\IR}{\mathbb{R}}
\newcommand{\ZZ}{\mathbb{Z}}

\newcommand{\ph}{\phantom}

\title{\vspace{-5mm}Cubic graphs, their Ehrhart quasi-polynomials, \\
and a scissors congruence phenomenon\thanks{FAPESP 13/03447-6 and 15/10323-7, CNPq 456792/2014-7 and 452507/2016-2, CAPES PROEX.}}

\author{\hspace{1cm} Cristina G.~Fernandes
	\thanks{Instituto de Matem\'atica e Estat\'istica, Universidade de S\~ao Paulo, 05508-090 S\~ao Paulo, Brazil.
	(\texttt{cris@ime.usp.br}, \texttt{coelho@ime.usp.br}, \texttt{srobins@ime.usp.br}).}
\and Jos\'e C.~de~Pina\footnotemark[2] \hspace{1cm}
\and Jorge~Luis Ram\'irez~Alfons\'in 
        \thanks{IMAG, Univ. Montpellier, CNRS, Montpellier, France. 
        (\texttt{jorge.ramirez-alfonsin@umontpellier.fr}).} 
\and Sinai Robins\footnotemark[2]
}

\begin{document}

\maketitle

\begin{abstract}
  The scissors congruence conjecture for the unimodular group 
  is an analogue of Hilbert's third problem, for the equidecomposability of polytopes.
  Liu and Osserman studied the Ehrhart quasi-polynomials of polytopes
  naturally associated to graphs whose vertices have degree one or three.
  In this paper, we prove the scissors congruence conjecture, posed by Haase and McAllister, 
  for this class of polytopes. 
  The key ingredient in the proofs is the nearest neighbor interchange on graphs
  and a naturally arising piecewise unimodular transformation.
\end{abstract}


\section{Introduction}\label{sec:intro}

A \emph{cubic graph} is a graph whose vertices have degree three and a
\emph{$\{1,3\}$-graph} is a graph whose vertices have degree either one or three.
The graphs are allowed to have loops and parallel edges.
Motivated by a result of Mochizuki~\cite{Mochizuki1996}, 
Liu and Osserman~\cite{LiuO2006} associated a polytope $\calP_G$ 
to each $\{1,3\}$-graph $G$ and studied its 
Ehrhart quasi-polynomial. 

For each degree three vertex $v$ of a $\{1,3\}$-graph $G=(V,E)$, 
let $a$,~$b$, and $c$ be the three edges incident to $v$.  
Denote by~$S(v)$ the linear system consisting of a perimeter inequality 
and three metric inequalities defined on the variables $w_a$, $w_b$, 
and $w_c$ as follows:
\[  w_a + w_b + w_c \leq 1 \] \vspace{-12mm}
\begin{eqnarray*}
  w_a & \leq & w_b + w_c \\
  w_b & \leq & w_a + w_c \\
  w_c & \leq & w_a + w_b.
\end{eqnarray*}
From the three metric inequalities, one can immediately conclude that 
$w_a$, $w_b$, and $w_c$ are nonnegative.  
Consider the union of all the linear systems $S(v)$, 
taken over all degree three vertices $v$ of~$G$.
The polytope $\calP_G$ is defined by the set of 
all real solutions for this linear system.

Let ${E = \{1,\ldots,m\}}$ and  $w: E \rightarrow \IR$ be a
weight function defined on the edges of~$G$.
We use the vector notation $w = (w_{1},\ldots,w_{m}) \in \IR^m$.
In particular, when~$w$ is a solution for the linear system
defining~$\calP_G$, we write $w \in \calP_G$.

Given a rational polytope~$\calP$, Eugène Ehrhart defined the function 
$L_{\calP}(t) := |t\calP \cap {\mathbb Z}^m|$, which is the 
number of lattice points in the closed \emph{dilated polytope} $t\calP$, 
for a nonnegative integer parameter $t$.
Ehrhart showed that this function is a polynomial in~$t$ when~$\calP$
is an integral polytope.  More generally, if~$\calP$ is a rational
polytope, the function $L_{\calP}(t)$
is a quasi-polynomial whose period is closely related
to the denominators appearing in the coordinates of the vertices 
of $\calP$~\cite{ehrhartbook, niceperson}.

Liu and Osserman conjectured (\cite[Conj.~4.2]{LiuO2006}) 
that polytopes associated to connected $\{1,3\}$-graphs with the
same number of vertices and edges have the \emph{same} Ehrhart
quasi-polynomial.
They partially proved their conjecture, by showing that these 
quasi-polynomials coincide for all nonnegative odd values of the dilation parameter~$t$.
In 2013, Wakabayashi~\cite[Thm.~A(ii)]{Wakabayashi2013} proved their conjecture.

An ingredient in Wakabayashi's proof for Liu and Osserman's conjecture is a
local transformation performed in $\{1,3\}$-graphs.
This transformation is called an $A$-move by Wakabayashi and it is also known as
a {\em nearest neighbor interchange} (NNI).
The NNI has been studied mainly for binary trees~\cite{CulikW1982, Robinson1971},
cubic graphs~\cite{Tsukui1996}, 
and $\{1,3\}$-graphs~\cite{Wakabayashi2013}. 
We present the following general result for connected graphs with the
same degree sequence, which might of interest on its own.
We refer to a vertex of degree one in a graph simply as a \emph{leaf}. 
An edge is \emph{external} if it is incident to a leaf, otherwise it is \emph{internal}. 

\begin{thm}\label{thm:connected}
  Let $G$ and $G'$ be connected graphs with the same degree sequence and the same
  set of external edges. Then
  \begin{enumerate}[(a)]
  \item $G$ can be transformed into $G'$ through a series of NNI moves. \label{thm:connected:a}
  \item One can choose a spanning tree in $G$ and a spanning tree in $G'$ 
    and require that all the pivots of the NNI moves are internal edges of both
    of these spanning
    trees.~\label{thm:connected:b}
  \end{enumerate}  
\end{thm}

In particular, for the special case of $\{1,3\}$-graphs, the proof provided
for Theorem~\ref{thm:connected}(a) constitutes an alternative graph theoretic proof
of a proposition by Wakabayashi~\cite[Prop.~6.2]{Wakabayashi2013}.

One of the concerns of this paper is on the scissors congruence conjecture for the unimodular group, 
which is an analogue of Hilbert's third problem (equidecomposability).
Concretely, this was stated as the following question by Haase and McAllister~\cite{haase2008}.
An integral matrix~$U$ is~\emph{unimodular} if it has  determinant~$\pm 1$.
An \emph{affine unimodular transformation} 
is defined by $x \rarr Ux+b$, where $U$ is a unimodular matrix and $b$ is a real vector. 

\begin{question}\!\!\textnormal{\cite[Question~4.1]{haase2008}}
  \label{qst:haase}
  Suppose that $\calP$ and $\calP'$ are polytopes with the same Ehrhart quasi-polynomial.
  Is it true that there is
  a decomposition of~$\calP$ into relatively open simplices $\calP^1,\ldots,\calP^k$
  and affine unimodular transformations $U^1,\ldots,U^k$
  such that $\calP'$ is the disjoint union of $U^1(\calP^1),\ldots,U^r(\calP^r)$?
\end{question}

We show that for polytopes associated to $\{1,3\}$-graphs such a scissors congruence decomposition holds. 
Namely, we have the following.

\begin{thm}\label{thm:decomposition}
  Let $G$ and $G'$ be two connected $\{1,3\}$-graphs with the same number of vertices and edges.
  Then there is
  a dissection of $\calP_G$ into smaller polytopes $\calP_G^1,\ldots,\calP_G^k$
  and affine unimodular transformations $U^1,\ldots,U^k$
  such that $\calP_{G'}$ is the \ignore{disjoint} union of $U^1(\calQ_G^1),\ldots,U^k(\calQ_G^k)$.
\end{thm}

The proof of Theorem~\ref{thm:decomposition} relies on a piecewise unimodular transformation
associated to a weighted version of the NNI move. 

The rational polytope  $\calP_G$, associated to a $\{1,3\}$-graph $G$, enjoys some fascinating symmetry.
Linke~\cite{Linke2011} considered the extension of $L_{\calP}(t)$ for all nonnegative real numbers~$t$. 
Royer~\cite{Royer2017} defined a polytope to be \emph{semi-reflexive} if $L_{\calP}(s) = L_{\calP}(\floor{s})$ for every nonnegative real number~$s$.
One can verify that $\calP_G$ is semi-reflexive. 
A polytope is \emph{reflexive} if it is integral, the origin is in its interior, and it is semi-reflexive~\cite{niceperson}.
We prove the following.

\begin{thm}\label{thm:reflexive}
For each $\{1,3\}$-graph $G$, the polytope $4\calP_G{-}\mathbbm{1}$
is reflexive.
\end{thm}

The paper is organized as follows.
Section~\ref{sec:nni} contains the results involving the NNI move in graphs
with the same degree sequence, including Theorem~\ref{thm:connected}, while 
Section~\ref{sec:wnni} discusses the extension of the NNI move to weighted graphs.
Section~\ref{sec:decomp} describes the unimodular decomposition of~$\calP_G$ 
and presents the proof of Theorem~\ref{thm:decomposition}.
In Section~\ref{sec:reflexivity}, we prove Theorem~\ref{thm:reflexive}.
Finally, Section~\ref{sec:remarks} contains some concluding remarks.



\section{Nearest neighbor interchange}\label{sec:nni}

With an eye towards Section~\ref{sec:wnni}, where we consider weighted graphs, 
we think of a graph as defined by Bondy and Murty~\cite{BondyM2008}.
A graph $G$ is an ordered pair $(V,E)$ consisting of a set~$V$ 
of vertices and a set $E$, disjoint from $V$, of edges, together
with an \emph{incidence function} $\psi_G$ that associates with each edge 
of $G$ an unordered pair of (not necessarily distinct) vertices of $G$.  
The \emph{degree sequence} of $G$ is the monotonic nonincreasing sequence 
consisting of the degrees of its vertices. 

A \emph{nearest neighbor interchange (NNI)} is a local move performed 
in $G$ on a trail~$W$ of length three.  
This move interchanges one end of the two extreme edges of~$W$ 
on the central edge (Figure~\ref{fig:nni}). 
We refer to the central edge of $W$ as the \emph{pivot} of the NNI move.
The result of the move is another graph~$G'$ on the same number of 
connected components, with the same degree sequence.  
We consider the graph $G'$ as having the same set of vertices and edges, 
that is, ${G'= (V,E)}$, and only the incidence function $\psi_{G'}$ is adjusted 
accordingly.  Intuitively, one can think of the edges as sticks that are being moved from~$G$ to $G'$. 

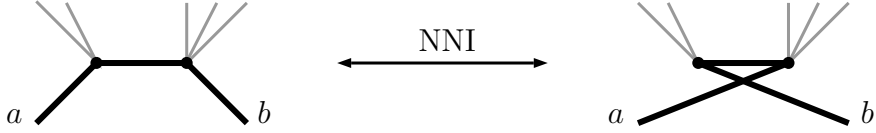
\begin{figure}[h]
  \centering
{
\begin{pspicture}(0,-0.85503983)(12.02,0.85503983)
\definecolor{colour0}{rgb}{0.6,0.6,0.6}
\psline[linecolor=black, linewidth=0.04, arrowsize=0.05291667cm 2.0,arrowlength=1.4,arrowinset=0.0]{<->}(4.4,0.020897675)(7.2,0.020897675)
\rput[bl](5.48,0.18089767){NNI}
\psline[linecolor=colour0, linewidth=0.04](0.4,0.8408977)(1.2,0.040897675)(0.8,0.8208977)
\psline[linecolor=colour0, linewidth=0.04](2.4,0.040897675)(2.8,0.8208977)
\psline[linecolor=colour0, linewidth=0.04](9.2,0.020897675)(8.4,0.8208977)
\psdots[linecolor=black, dotsize=0.16](1.2,0.020897675)
\psline[linecolor=colour0, linewidth=0.04](2.4,0.8208977)(2.4,0.020897675)(3.2,0.8208977)
\psdots[linecolor=black, dotsize=0.16](2.4,0.020897675)
\psline[linecolor=black, linewidth=0.08](0.4,-0.7791023)(1.2,0.020897675)(2.4,0.020897675)(3.2,-0.7791023)
\psline[linecolor=black, linewidth=0.08](8.4,-0.7791023)(10.4,0.020897675)
\psline[linecolor=black, linewidth=0.08](9.2,0.020897675)(11.2,-0.7791023)
\psline[linecolor=black, linewidth=0.08](9.2,0.020897675)(10.4,0.020897675)
\psline[linecolor=colour0, linewidth=0.04](9.2,0.020897675)(8.8,0.8208977)
\psline[linecolor=colour0, linewidth=0.04](10.4,0.8208977)(10.4,0.020897675)(10.8,0.8208977)
\psline[linecolor=colour0, linewidth=0.04](10.4,0.020897675)(11.2,0.8208977)
\psdots[linecolor=black, dotsize=0.16](9.2,0.020897675)
\psdots[linecolor=black, dotsize=0.16](10.4,0.020897675)
\rput[bl](0.0,-0.7791023){$a$}
\rput[bl](3.34,-0.7791023){$b$}
\rput[bl](8.0,-0.7791023){$a$}
\rput[bl](11.36,-0.7791023){$b$}
\end{pspicture}
}
  \caption{An NNI move on the trail marked in bold. One of the incidences of the edges~$a$ and $b$ were interchanged.}
  \label{fig:nni}
\end{figure}

We think of an NNI move as a function $\gamma_W$ that associates 
the graph~$G$ to the graph~$G'$.  In symbols, $\gamma_W(G) = G'$.  
Observe that $W$ is also a trail in~$G'$ and $\gamma_{W}(G') = G$. 

The well-known rotation, used in data structures to balance binary trees,
is a particular NNI move, performed on a $\{1,3\}$-tree (Figure~\ref{fig:rotation}).

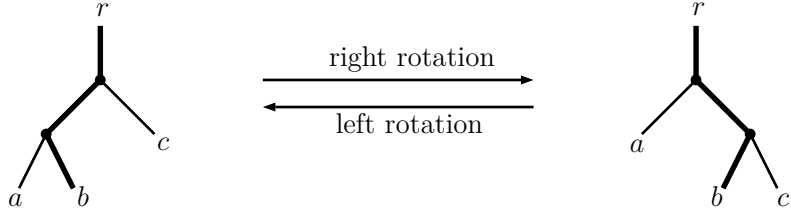
\begin{figure}[h]
  \centering
  \scalebox{.9}{
{
\begin{pspicture}(0,-1.5617871)(11.98,1.5617871)
\psline[linecolor=black, linewidth=0.04](9.36,-0.42584962)(10.16,0.3741504)(10.96,-0.42584962)(10.56,-1.2258496)
\psline[linecolor=black, linewidth=0.04, arrowsize=0.05291667cm 2.0,arrowlength=1.4,arrowinset=0.0]{->}(3.76,0.3741504)(7.76,0.3741504)
\psline[linecolor=black, linewidth=0.04, arrowsize=0.05291667cm 2.0,arrowlength=1.4,arrowinset=0.0]{->}(7.76,-0.025849609)(3.76,-0.025849609)
\rput[bl](4.84,-0.4058496){left rotation}
\rput[bl](4.74,0.4941504){right rotation}
\rput[bl](1.3,1.3141503){$r$}
\rput[bl](2.2,-0.6658496){$c$}
\rput[bl](1.02,-1.4858496){$b$}
\rput[bl](0.0,-1.4858496){$a$}
\psline[linecolor=black, linewidth=0.04](0.16,-1.2258496)(0.56,-0.42584962)(0.96,-1.2258496)
\psline[linecolor=black, linewidth=0.04](1.36,0.3741504)(2.16,-0.42584962)
\psline[linecolor=black, linewidth=0.08](1.36,1.1741503)(1.36,0.3741504)(0.56,-0.42584962)
\rput[bl](10.1,1.3341504){$r$}
\rput[bl](11.36,-1.4858496){$c$}
\rput[bl](10.38,-1.4858496){$b$}
\rput[bl](9.18,-0.6858496){$a$}
\psline[linecolor=black, linewidth=0.08](10.16,0.3741504)(10.16,1.1741503)
\psline[linecolor=black, linewidth=0.04](10.96,-0.42584962)(11.36,-1.2258496)
\psdots[linecolor=black, dotsize=0.16](0.56,-0.42584962)
\psdots[linecolor=black, dotsize=0.16](1.36,0.3741504)
\psdots[linecolor=black, dotsize=0.16](10.16,0.3741504)
\psdots[linecolor=black, dotsize=0.16](10.96,-0.42584962)
\psline[linecolor=black, linewidth=0.08](0.56,-0.42584962)(0.96,-1.2258496)
\psline[linecolor=black, linewidth=0.08](10.16,0.3741504)(10.96,-0.42584962)
\psline[linecolor=black, linewidth=0.08](10.96,-0.42584962)(10.56,-1.2258496)
\end{pspicture}
}}
  \caption{Right and left rotations applied to a binary tree.}
  \label{fig:rotation}
\end{figure}

An NNI move does not affect the incidence to leaves, thus it preserves
the partition of the edge set into internal and external edges.

Culik and Wood~\cite[Thm.~2.4]{CulikW1982} proved that any two 
$\{1,3\}$-trees with $\ell$ (labelled) leaves can be transformed 
into one another through a finite series of NNI moves.  
They additionally gave an upper bound of~$4\ell-12+4\ell \log_2 (\ell/3)$ 
on the number of NNI moves needed for this transformation.
In this section, first we extend Culik and Wood's theorem to trees 
with the same degree sequence (Lemma~\ref{lem:treesdegreeseq}), 
which we then use to extend their result further to connected graphs 
with the same degree sequence (Theorem~\ref{thm:connected}). 

A \emph{caterpillar} is a tree for which the removal of all leaves results in a path,
called its \emph{central path}, or in the empty graph.  For the later, we define that the central 
path is empty.  

\begin{lem}\label{lem:caterpillar}
  Any tree can be transformed into a caterpillar 
  with the same degree sequence through a series of NNI moves. 
\end{lem}
\begin{proof}
  Let $T$ be a tree.  
  If $T$ is a caterpillar, there is nothing to prove. 
  So, we may assume~$T$ is not a caterpillar. 
  Let $P$ be a longest path in $T$ and $uv$ an internal edge of $T$ not in~$P$ such that $u$ is a vertex in $P$.
  The vertex $u$ has two neighbors in~$P$, otherwise~$P$ would not be a longest path, as $v$ is not in~$P$.
  Let $w$ and $z$ be the two neighbors of $u$ in~$P$. 
  Let $W$ be a trail with edges $wu$, $uv$, and $vv'$, where $v'$ is a neighbor of $v$ other than $u$.
  Perform an NNI on the trail $W$ as in Figure~\ref{fig:NNI4trees}, to insert $v$ in $P$, 
  obtaining another tree~$T'$ with the same degree sequence and a path longer than $P$. 
  By repeating this process, we obtain a desired caterpillar after a finite number of NNI moves. 
\end{proof}

\begin{figure}[ht]
  \centering
  \scalebox{0.9}{
{
\begin{pspicture}(0,-1.4532415)(14.681667,1.4532415)
\definecolor{colour1}{rgb}{1.0,0.0,0.2}
\psline[linecolor=red, linewidth=0.08](0.74083346,1.0256047)(1.9408334,1.0256047)(1.9408334,-0.1743953)(1.5408335,-1.3743953)
\psline[linecolor=black, linewidth=0.04](1.9408334,-0.1743953)(1.9408334,-1.3743953)
\psline[linecolor=black, linewidth=0.04](1.9408334,-0.1743953)(2.3408334,-1.3743953)
\psline[linecolor=black, linewidth=0.04](1.9408334,1.0256047)(3.1408334,1.0256047)
\psline[linecolor=black, linewidth=0.04, linestyle=dashed, dash=0.17638889cm 0.10583334cm](3.1408334,1.0256047)(3.9408333,1.0256047)
\psline[linecolor=black, linewidth=0.04, linestyle=dashed, dash=0.17638889cm 0.10583334cm](0.74083346,1.0256047)(-0.059166566,1.0256047)
\rput[bl](2.2808335,-0.37439528){$v$}
\rput[bl](1.7608334,1.2056047){$u$}
\rput[bl](0.44083345,1.2056047){$w$}
\rput[bl](3.0208335,1.2256047){$z$}
\rput[bl](5.2408333,1.2056047){$w$}
\rput[bl](7.820833,1.2256047){$z$}
\psline[linecolor=black, linewidth=0.04](6.7408333,-0.1743953)(6.7408333,-1.3743953)
\psline[linecolor=black, linewidth=0.04](6.7408333,-0.1743953)(7.1408334,-1.3743953)
\psline[linecolor=black, linewidth=0.04](6.7408333,1.0256047)(7.9408336,1.0256047)
\psline[linecolor=black, linewidth=0.04, linestyle=dashed, dash=0.17638889cm 0.10583334cm](7.9408336,1.0256047)(8.740833,1.0256047)
\psline[linecolor=black, linewidth=0.04, linestyle=dashed, dash=0.17638889cm 0.10583334cm](5.5408335,1.0256047)(4.7408333,1.0256047)
\rput[bl](7.0808334,-0.37439528){$v$}
\rput[bl](6.5608335,1.2056047){$u$}
\psline[linecolor=red, linewidth=0.08](6.7408333,1.0256047)(6.7408333,1.0256047)(6.7408333,-0.1743953)
\psline[linecolor=red, linewidth=0.08](5.5408335,1.0256047)(6.7408333,-0.1743953)
\psbezier[linecolor=red, linewidth=0.08](6.3408337,-1.3743953)(5.9408336,-1.3743953)(5.9408336,0.6256047)(6.7408333,1.0256047058105469)
\psline[linecolor=red, linewidth=0.08](12.740833,1.0256047)(12.740833,-0.1743953)
\psline[linecolor=colour1, linewidth=0.08](10.340834,1.0256047)(12.740833,1.0256047)
\psline[linecolor=black, linewidth=0.04](11.540833,1.0256047)(11.940833,-0.1743953)
\psline[linecolor=black, linewidth=0.04](12.740833,1.0256047)(13.940833,1.0256047)
\psline[linecolor=black, linewidth=0.04](11.540833,1.0256047)(11.540833,-0.1743953)
\psline[linecolor=black, linewidth=0.04, linestyle=dashed, dash=0.17638889cm 0.10583334cm](13.940833,1.0256047)(14.740833,1.0256047)
\psline[linecolor=black, linewidth=0.04, linestyle=dashed, dash=0.17638889cm 0.10583334cm](10.340834,1.0256047)(9.540833,1.0256047)
\rput[bl](13.820833,1.2256047){$z$}
\rput[bl](12.680834,1.2056047){$u$}
\rput[bl](11.360833,1.2256047){$v$}
\rput[bl](10.060833,1.2256047){$w$}
\psdots[linecolor=black, dotsize=0.16](0.74083346,1.0256047)
\psdots[linecolor=black, dotsize=0.16](1.9408334,1.0256047)
\psdots[linecolor=black, dotsize=0.16](1.9408334,-0.1743953)
\psdots[linecolor=black, dotsize=0.16](1.5408335,-1.3743953)
\psdots[linecolor=black, dotsize=0.16](3.1408334,1.0256047)
\psdots[linecolor=black, dotsize=0.16](5.5408335,1.0256047)
\psdots[linecolor=black, dotsize=0.16](6.7408333,1.0256047)
\psdots[linecolor=black, dotsize=0.16](7.9408336,1.0256047)
\psdots[linecolor=black, dotsize=0.16](6.7408333,-0.1743953)
\psdots[linecolor=black, dotsize=0.16](6.3408337,-1.3743953)
\psdots[linecolor=black, dotsize=0.16](10.340834,1.0256047)
\psdots[linecolor=black, dotsize=0.16](11.540833,1.0256047)
\psdots[linecolor=black, dotsize=0.16](12.740833,1.0256047)
\psdots[linecolor=black, dotsize=0.16](13.940833,1.0256047)
\psdots[linecolor=black, dotsize=0.16](12.740833,-0.1743953)
\end{pspicture}
}}
  \caption{An NNI to insert $v$ in the central path.}
  \label{fig:NNI4trees}
\end{figure}
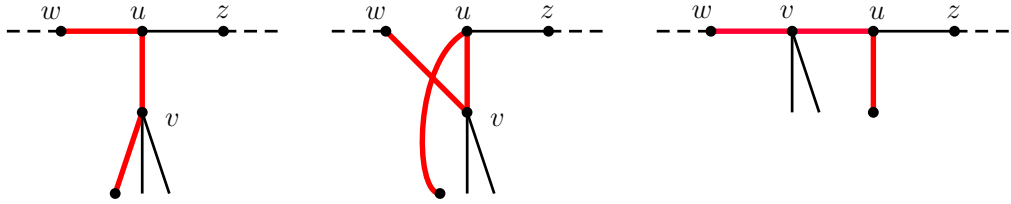

The \emph{spine} of a caterpillar is the sequence of the degrees of
the vertices in the central path.  We say the caterpillar is
\emph{ordered} if its spine is a monotonic nonincreasing sequence
(Figure~\ref{fig:caterpillar}).

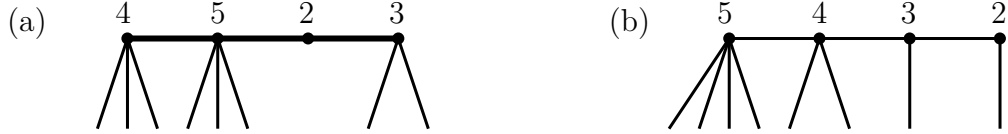
\begin{figure}[ht]
  \centering
{
\begin{pspicture}(0,-0.8041562)(13.3,0.8041562)
\rput[bl](0.0,0.4069296){(a)}
\psdots[linecolor=black, dotsize=0.16](5.2,0.4069296)
\psdots[linecolor=black, dotsize=0.16](4.0,0.4069296)
\psdots[linecolor=black, dotsize=0.16](2.8,0.4069296)
\psdots[linecolor=black, dotsize=0.16](1.6,0.4069296)
\rput[bl](5.08,0.6069296){3}
\rput[bl](3.88,0.6069296){2}
\rput[bl](2.7,0.6069296){5}
\rput[bl](1.44,0.6069296){4}
\psline[linecolor=black, linewidth=0.04](5.2,0.4069296)(5.6,-0.7930704)
\psline[linecolor=black, linewidth=0.04](5.2,0.4069296)(4.8,-0.7930704)
\psline[linecolor=black, linewidth=0.04](2.8,0.4069296)(3.2,-0.7930704)
\psline[linecolor=black, linewidth=0.04](2.8,0.4069296)(2.4,-0.7930704)
\psline[linecolor=black, linewidth=0.04](2.8,0.4069296)(2.8,-0.7930704)
\psline[linecolor=black, linewidth=0.04](1.6,0.4069296)(2.0,-0.7930704)
\psline[linecolor=black, linewidth=0.04](1.6,0.4069296)(1.2,-0.7930704)
\psline[linecolor=black, linewidth=0.04](1.6,0.4069296)(1.6,-0.7930704)
\psline[linecolor=black, linewidth=0.08](1.6,0.4069296)(5.2,0.4069296)
\psline[linecolor=black, linewidth=0.04](9.6,0.4069296)(13.2,0.4069296)
\psline[linecolor=black, linewidth=0.04](9.6,0.4069296)(9.6,-0.7930704)
\psline[linecolor=black, linewidth=0.04](9.6,0.4069296)(9.2,-0.7930704)
\psline[linecolor=black, linewidth=0.04](9.6,0.4069296)(10.0,-0.7930704)
\psline[linecolor=black, linewidth=0.04](10.8,0.4069296)(10.4,-0.7930704)
\psline[linecolor=black, linewidth=0.04](10.8,0.4069296)(11.2,-0.7930704)
\psline[linecolor=black, linewidth=0.04](13.2,0.4069296)(13.2,-0.7930704)
\rput[bl](9.44,0.6069296){5}
\rput[bl](10.7,0.6069296){4}
\rput[bl](11.88,0.6069296){3}
\rput[bl](13.08,0.6069296){2}
\psdots[linecolor=black, dotsize=0.16](9.6,0.4069296)
\psdots[linecolor=black, dotsize=0.16](10.8,0.4069296)
\psdots[linecolor=black, dotsize=0.16](12.0,0.4069296)
\psdots[linecolor=black, dotsize=0.16](13.2,0.4069296)
\rput[bl](8.0,0.4069296){(b)}
\psline[linecolor=black, linewidth=0.04](9.6,0.4069296)(8.8,-0.7930704)
\psline[linecolor=black, linewidth=0.04](12.0,0.4069296)(12.0,-0.7930704)
\end{pspicture}
}
  \caption{(a) A caterpillar with spine $(4,5,2,3)$ and central path highlighted in bold.
           (b)~An ordered caterpillar.}
  \label{fig:caterpillar}
\end{figure}

\begin{lem}\label{lem:orderedcaterpillar}
  Any caterpillar can be transformed into an ordered caterpillar 
  with the same degree sequence through a series of NNI moves. 
\end{lem}
\begin{proof}
  An NNI can be used to swap any two adjacent vertices in the central path of a caterpillar, 
  as in Figure~\ref{fig:cat-NNI-t1}(a). 
  So we use an NNI to decrease, one by one, the number of inversions 
  in the spine of a caterpillar until we obtain an ordered caterpillar.
\end{proof}

\begin{figure}[ht]
  \centering
{
\begin{pspicture}(0,-3.441787)(13.940834,3.441787)
\definecolor{colour1}{rgb}{1.0,0.0,0.2}
\psline[linecolor=colour1, linewidth=0.08](1.6,3.0141504)(5.2,3.0141504)
\psline[linecolor=black, linewidth=0.04](4.0,3.0141504)(4.4,1.8141503)
\psline[linecolor=black, linewidth=0.04](4.0,3.0141504)(3.6,1.8141503)
\psline[linecolor=black, linewidth=0.04](2.8,3.0141504)(3.2,1.8141503)
\psline[linecolor=black, linewidth=0.04](4.0,3.0141504)(4.0,1.8141503)
\psline[linecolor=black, linewidth=0.04](2.8,3.0141504)(2.4,1.8141503)
\psline[linecolor=black, linewidth=0.04, linestyle=dashed, dash=0.17638889cm 0.10583334cm](5.2,3.0141504)(6.0,3.0141504)
\psline[linecolor=black, linewidth=0.04, linestyle=dashed, dash=0.17638889cm 0.10583334cm](1.6,3.0141504)(0.8,3.0141504)
\rput[bl](5.08,3.2141504){$z$}
\rput[bl](3.94,3.2141504){$v$}
\rput[bl](2.62,3.1941504){$u$}
\rput[bl](1.28,3.2141504){$w$}
\psline[linecolor=black, linewidth=0.04](4.0,0.6141504)(4.4,-0.5858496)
\psline[linecolor=black, linewidth=0.04](4.0,0.6141504)(3.6,-0.5858496)
\psline[linecolor=black, linewidth=0.04](2.8,0.6141504)(3.2,-0.5858496)
\psline[linecolor=black, linewidth=0.04](4.0,0.6141504)(4.0,-0.5858496)
\psline[linecolor=black, linewidth=0.04](2.8,0.6141504)(2.4,-0.5858496)
\psline[linecolor=black, linewidth=0.04, linestyle=dashed, dash=0.17638889cm 0.10583334cm](5.2,0.6141504)(6.0,0.6141504)
\psline[linecolor=black, linewidth=0.04, linestyle=dashed, dash=0.17638889cm 0.10583334cm](1.6,0.6141504)(0.8,0.6141504)
\rput[bl](5.08,0.8141504){$z$}
\rput[bl](3.96,0.8141504){$v$}
\rput[bl](2.62,0.7941504){$u$}
\psline[linecolor=colour1, linewidth=0.08](1.6,-1.7858496)(5.2,-1.7858496)
\rput[bl](1.32,-1.5858496){$w$}
\rput[bl](2.62,-1.5858496){$v$}
\rput[bl](3.94,-1.6058496){$u$}
\rput[bl](5.08,-1.5858496){$z$}
\psline[linecolor=black, linewidth=0.04, linestyle=dashed, dash=0.17638889cm 0.10583334cm](1.6,-1.7858496)(0.8,-1.7858496)
\psline[linecolor=black, linewidth=0.04, linestyle=dashed, dash=0.17638889cm 0.10583334cm](5.2,-1.7858496)(6.0,-1.7858496)
\psline[linecolor=black, linewidth=0.04](2.8,-1.7858496)(2.4,-2.9858496)
\psline[linecolor=black, linewidth=0.04](2.8,-1.7858496)(2.8,-2.9858496)
\psline[linecolor=black, linewidth=0.04](2.8,-1.7858496)(3.2,-2.9858496)
\psline[linecolor=black, linewidth=0.04](4.0,-1.7858496)(3.6,-2.9858496)
\psline[linecolor=black, linewidth=0.04](4.0,-1.7858496)(4.4,-2.9858496)
\psbezier[linecolor=red, linewidth=0.08](1.6,0.6141504)(1.6,1.4141504)(4.0,1.4141504)(4.0,0.614150390625)
\psline[linecolor=colour1, linewidth=0.08](2.8,0.6141504)(4.0,0.6141504)
\psbezier[linecolor=red, linewidth=0.08](2.8,0.6141504)(2.8,-0.1858496)(5.2,-0.1858496)(5.2,0.614150390625)
\rput[bl](1.3,0.7941504){$w$}
\rput[bl](0.0,3.0141504){(a)}
\rput[bl](8.0,3.0141504){(b)}
\rput[bl](11.64,-0.96584964){$d$}
\rput[bl](11.12,-0.96584964){$e$}
\rput[bl](10.64,-0.96584964){$a$}
\rput[bl](10.24,-0.96584964){$c$}
\rput[bl](12.12,-0.96584964){$b$}
\psline[linecolor=black, linewidth=0.04](12.0,0.6141504)(12.24,-0.5858496)
\psline[linecolor=black, linewidth=0.04](12.0,0.6141504)(11.76,-0.5858496)
\psline[linecolor=black, linewidth=0.04](10.8,0.6141504)(11.2,-0.5858496)
\psline[linecolor=black, linewidth=0.04](10.8,0.6141504)(10.8,-0.5858496)
\psline[linecolor=black, linewidth=0.04](10.8,0.6141504)(10.4,-0.5858496)
\psline[linecolor=black, linewidth=0.04, linestyle=dashed, dash=0.17638889cm 0.10583334cm](13.2,0.6141504)(14.0,0.6141504)
\psline[linecolor=black, linewidth=0.04, linestyle=dashed, dash=0.17638889cm 0.10583334cm](9.6,0.6141504)(8.8,0.6141504)
\rput[bl](13.08,0.8141504){$z$}
\rput[bl](11.94,0.8141504){$v$}
\rput[bl](10.62,0.7941504){$u$}
\rput[bl](9.28,0.8141504){$w$}
\psline[linecolor=black, linewidth=0.04](9.6,0.6141504)(13.2,0.6141504)
\rput[bl](11.64,-3.3658495){$d$}
\rput[bl](11.12,-3.3658495){$c$}
\rput[bl](10.7,-3.3658495){$b$}
\rput[bl](10.24,-3.3658495){$a$}
\rput[bl](12.12,-3.3658495){$e$}
\psline[linecolor=black, linewidth=0.04](12.0,-1.7858496)(12.24,-2.9858496)
\psline[linecolor=black, linewidth=0.04](12.0,-1.7858496)(11.76,-2.9858496)
\psline[linecolor=black, linewidth=0.04](10.8,-1.7858496)(11.2,-2.9858496)
\psline[linecolor=black, linewidth=0.04](10.8,-1.7858496)(10.8,-2.9858496)
\psline[linecolor=black, linewidth=0.04](10.8,-1.7858496)(10.4,-2.9858496)
\psline[linecolor=black, linewidth=0.04, linestyle=dashed, dash=0.17638889cm 0.10583334cm](13.2,-1.7858496)(14.0,-1.7858496)
\psline[linecolor=black, linewidth=0.04, linestyle=dashed, dash=0.17638889cm 0.10583334cm](9.6,-1.7858496)(8.8,-1.7858496)
\rput[bl](13.08,-1.5858496){$z$}
\rput[bl](11.94,-1.5858496){$v$}
\rput[bl](10.62,-1.6058496){$u$}
\rput[bl](9.28,-1.5858496){$w$}
\psline[linecolor=black, linewidth=0.04](9.6,-1.7858496)(13.2,-1.7858496)
\psline[linecolor=red, linewidth=0.08](12.22,-0.5858496)(12.0,0.6141504)(10.8,0.6141504)(11.2,-0.5858496)(11.2,-0.5858496)(11.2,-0.5858496)(11.2,-0.5858496)
\rput[bl](11.64,1.4341503){$a$}
\rput[bl](11.12,1.4341503){$e$}
\rput[bl](10.64,1.4341503){$d$}
\rput[bl](10.24,1.4341503){$c$}
\rput[bl](12.24,1.4141504){$b$}
\psline[linecolor=black, linewidth=0.04](12.0,3.0141504)(12.24,1.8141503)
\psline[linecolor=black, linewidth=0.04](12.0,3.0141504)(11.76,1.8141503)
\psline[linecolor=black, linewidth=0.04](10.8,3.0141504)(11.2,1.8141503)
\psline[linecolor=black, linewidth=0.04](10.8,3.0141504)(10.8,1.8141503)
\psline[linecolor=black, linewidth=0.04](10.8,3.0141504)(10.4,1.8141503)
\psline[linecolor=black, linewidth=0.04, linestyle=dashed, dash=0.17638889cm 0.10583334cm](13.2,3.0141504)(14.0,3.0141504)
\psline[linecolor=black, linewidth=0.04, linestyle=dashed, dash=0.17638889cm 0.10583334cm](9.6,3.0141504)(8.8,3.0141504)
\rput[bl](13.08,3.2141504){$z$}
\rput[bl](11.94,3.2141504){$v$}
\rput[bl](10.62,3.1941504){$u$}
\rput[bl](9.28,3.2141504){$w$}
\psline[linecolor=black, linewidth=0.04](9.6,3.0141504)(13.2,3.0141504)
\psline[linecolor=red, linewidth=0.08](11.78,1.8141503)(12.0,3.0141504)(10.8,3.0141504)(10.8,1.8141503)(10.8,1.8141503)
\psdots[linecolor=black, dotsize=0.16](1.6,3.0141504)
\psdots[linecolor=black, dotsize=0.16](1.6,0.6141504)
\psdots[linecolor=black, dotsize=0.16](1.6,-1.7858496)
\psdots[linecolor=black, dotsize=0.16](2.8,-1.7858496)
\psdots[linecolor=black, dotsize=0.16](2.8,0.6141504)
\psdots[linecolor=black, dotsize=0.16](2.8,3.0141504)
\psdots[linecolor=black, dotsize=0.16](4.0,3.0141504)
\psdots[linecolor=black, dotsize=0.16](5.2,3.0141504)
\psdots[linecolor=black, dotsize=0.16](5.2,0.6141504)
\psdots[linecolor=black, dotsize=0.16](4.0,0.6141504)
\psdots[linecolor=black, dotsize=0.16](4.0,-1.7858496)
\psdots[linecolor=black, dotsize=0.16](5.2,-1.7858496)
\psdots[linecolor=black, dotsize=0.16](10.8,3.0141504)
\psdots[linecolor=black, dotsize=0.16](9.6,3.0141504)
\psdots[linecolor=black, dotsize=0.16](12.0,3.0141504)
\psdots[linecolor=black, dotsize=0.16](13.2,3.0141504)
\psdots[linecolor=black, dotsize=0.16](13.2,0.6141504)
\psdots[linecolor=black, dotsize=0.16](13.2,-1.7858496)
\psdots[linecolor=black, dotsize=0.16](12.0,-1.7858496)
\psdots[linecolor=black, dotsize=0.16](12.0,0.6141504)
\psdots[linecolor=black, dotsize=0.16](10.8,0.6141504)
\psdots[linecolor=black, dotsize=0.16](10.8,-1.7858496)
\psdots[linecolor=black, dotsize=0.16](9.6,0.6141504)
\psdots[linecolor=black, dotsize=0.16](9.6,-1.7858496)
\end{pspicture}
}
  \caption{(a) An NNI to swap two adjacent vertices in the central path of a caterpillar.
           (b) NNIs to swap labels on leaves hanging from adjacent vertices in the central path.}
  \label{fig:cat-NNI-t1}
\end{figure}
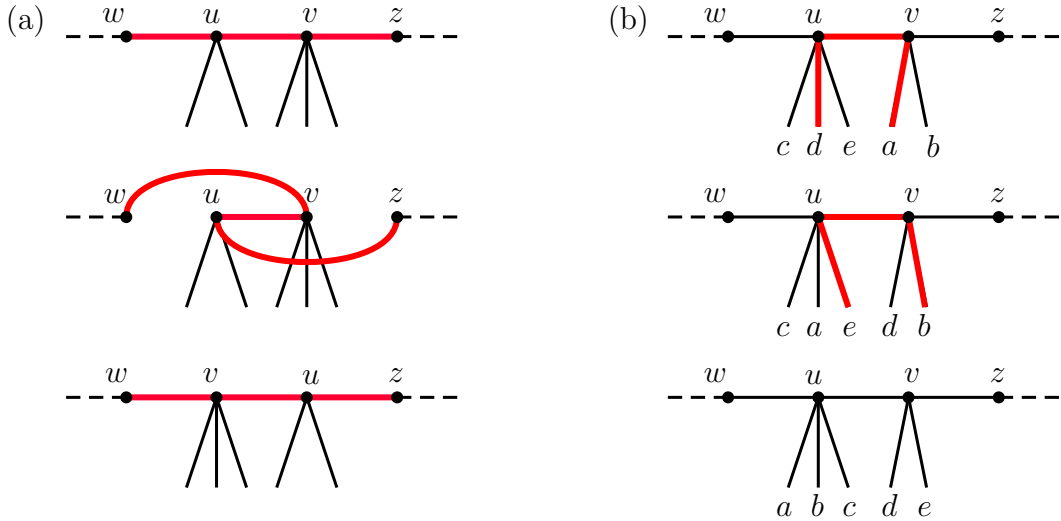

\begin{lem}\label{lem:sortlabels}
  The external edges of a caterpillar can be sorted arbitrarily through a series of NNI moves. 
\end{lem}
\begin{proof}
  An NNI can be used to swap any two external edges incident to two adjacent vertices 
  in the central path of a caterpillar, as in Figure~\ref{fig:cat-NNI-t1}(b). 
\end{proof}


The next lemma is an extension of previous works in the
literature~\cite{CulikW1982, Robinson1971, Tsukui1996,
  Wakabayashi2013} that might be of independent interest.

\begin{lem}\label{lem:treesdegreeseq}
  Any two trees with the same degree sequence and the same set of external edges 
  can be transformed into one another through a series of NNI moves.  
\end{lem}

\begin{proof}
  Let $T$ and $T'$ be two trees with the same degree sequence and the same set of external edges. 
  Using Lemmas~\ref{lem:caterpillar} and \ref{lem:orderedcaterpillar}, we obtain  
  a series of NNI moves that transforms~$T$ into an ordered caterpillar with the same degree 
  sequence and the same set of external edges of $T$.
  Similarly, we obtain another series for $T'$.
  Using Lemma~\ref{lem:sortlabels}, we extend the series of NNI moves for $T$ to sort the 
  external edges of the caterpillar coming from $T$ into the order they appear in the
  caterpillar coming from $T'$. 
  The composition of these two series of NNI moves, with the series for $T'$ inverted, 
  gives a series of NNI moves that transforms $T$ into $T'$. 
\end{proof}


Let $G$ be a connected graph that is not a tree, and let $e$
be an edge that is in a cycle.  The graph obtained from $G$ by
\emph{cutting} $e$ is the graph $G'$ resulting from the splitting of $e$
into two edges, each connecting one of the ends of $e$ to one
of two new leaves (Figure~\ref{fig:edgesplit}).

\begin{figure}[ht]
  \centering
  \psscalebox{1.0 1.0} 
{
\begin{pspicture}(0,-1.4394231)(10.557693,1.4394231)
\psline[linecolor=black, linewidth=0.04](1.6788461,1.360577)(0.07884613,-1.0394231)(3.278846,-1.0394231)(1.6788461,1.360577)(1.6788461,-0.23942307)(0.07884613,-1.0394231)
\psline[linecolor=black, linewidth=0.04](1.6788461,-0.23942307)(3.278846,-1.0394231)
\rput[bl](1.5588461,-1.4394231){$e$}
\rput[bl](0.07884613,0.96057695){$G$}
\psdots[linecolor=black, dotsize=0.16](0.07884613,-1.0394231)
\psdots[linecolor=black, dotsize=0.16](3.278846,-1.0394231)
\psdots[linecolor=black, dotsize=0.16](1.6788461,-0.23942307)
\psdots[linecolor=black, dotsize=0.16](1.6788461,1.360577)
\psdots[linecolor=black, dotsize=0.16](8.078846,-1.0394231)
\psdots[linecolor=black, dotsize=0.16](9.678846,-1.0394231)
\psdots[linecolor=black, dotsize=0.16](10.478847,-1.0394231)
\psdots[linecolor=black, dotsize=0.16](7.2788463,-1.0394231)
\psdots[linecolor=black, dotsize=0.16](8.878846,-0.23942307)
\psdots[linecolor=black, dotsize=0.16](8.878846,1.360577)
\rput[bl](7.2788463,0.96057695){$G'$}
\psline[linecolor=black, linewidth=0.04](10.478847,-1.0394231)(9.678846,-1.0394231)
\psline[linecolor=black, linewidth=0.04](7.2788463,-1.0394231)(8.078846,-1.0394231)
\psline[linecolor=black, linewidth=0.04](10.478847,-1.0394231)(8.878846,1.360577)(7.2788463,-1.0394231)(8.878846,-0.23942307)(8.878846,1.360577)
\psline[linecolor=black, linewidth=0.04](8.878846,-0.23942307)(10.478847,-1.0394231)
\end{pspicture}
}
  \caption{Graph $G'$ obtained from $G$ by cutting an edge $e$.}
  \label{fig:edgesplit}
\end{figure}
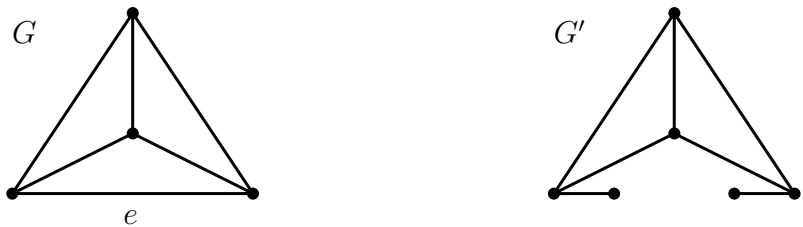



\begin{proof}[Proof of Theorem~\ref{thm:connected}]
  Let $n$ be the number of vertices and $m$ be the number edges 
  in $G$ and $G'$.
  The proof of the theorem is by induction on the dimension $r=m-n+1$ of the cycle space
  of~$G$ and~$G'$.
  If $r = 0$, then~$G$ and~$G'$ are trees and the theorem follows from Lemma~\ref{lem:treesdegreeseq}.
  Hence we may assume that~$r > 0$.
  
  Choose an edge in a cycle of $G$ and an edge in a cycle of $G'$.
  Let us denote these two edges by $e$.
  Let $H$ and $H'$ be the graphs obtained from~$G$ and $G'$, respectively,
  by cutting their edge $e$.
  The number of vertices in $H$ and $H'$ is~$n+2$ and the number of edges
  is $m+1$.
  Call $e'$ and $e''$ the two new edges in both $H$ and $H'$.
  Since~$e$ is in a cycle in~$G$ and in a cycle in~$G'$, the graphs $H$ and $H'$
  are connected and the dimension of their cycle space is~$r-1$.  
  Also, $H$ and $H'$ have the same degree sequence and the same 
  set of external edges.
  By induction, $H$ can be transformed into $H'$ 
  through a series of NNI moves. 
  The same series of NNI moves transforms~$G$ into~$G'$.
  Indeed, the set of external edges in all graphs obtained during the 
  application of this series of NNI moves, transforming $H$ into~$H'$,
  contain $e'$ and $e''$.  Glue $e'$ and $e''$ into an edge in each 
  of these graphs, obtaining a sequence of connected graphs that is the 
  result of a series of NNI moves starting at~$G$ and ending at~$G'$.
  This ends the proof of \eqref{thm:connected:a}.
  $$\begin{array}{cccccc}
  G  & & & & G \\
  \ \ \Downarrow {\scriptscriptstyle{\mathrm{cut}}} & & & &  \ \ \Downarrow {\scriptscriptstyle{\mathrm{cut}}} \\
  H = H_0& \substack{{\scriptscriptstyle \mathrm{NNI}}\\\longleftrightarrow} & \cdots & \substack{{\scriptscriptstyle \mathrm{NNI}}\\\longleftrightarrow} & H_k = H' \\
  \ \ \ \Downarrow {\scriptscriptstyle{\mathrm{glue}}} & & \ \ \ \ \Downarrow {\scriptscriptstyle{\mathrm{glue}}} & &  \ \ \ \Downarrow {\scriptscriptstyle{\mathrm{glue}}} \\
  G = G_0& \substack{{\scriptscriptstyle \mathrm{NNI}}\\\longleftrightarrow} & \cdots & \substack{{\scriptscriptstyle \mathrm{NNI}}\\\longleftrightarrow} & G_k = G' \\
  \end{array}$$

  Similarly, by induction in the cycle space of $G$ and $G'$, one can prove \eqref{thm:connected:b}.
  Indeed, it is sufficient to choose in each step of the induction an edge to be cut which is not
  in either of the spanning trees.
\end{proof}

For the case of cubic graphs, the proof of Theorem~\ref{thm:connected} provides 
an alternative proof for a theorem by Tsukui~\cite[Thm.~II]{Tsukui1996}, which 
refers to NNI moves as $\tilde{S}$-transformations, where the `S' stands for \emph{slide}. 
For the slightly more general case of $\{1,3\}$-graphs, the proof of Theorem~\ref{thm:connected} 
provides an alternative proof for a proposition by Wakabayashi~\cite[Prop.~6.2]{Wakabayashi2013},
which refers to NNI moves as $A$-moves.  
Wakabayashi's proof uses a topological pants decomposition for compact, 
oriented surfaces of finite genus.

\begin{figure}
  \centering
  \scalebox{0.70}{\input{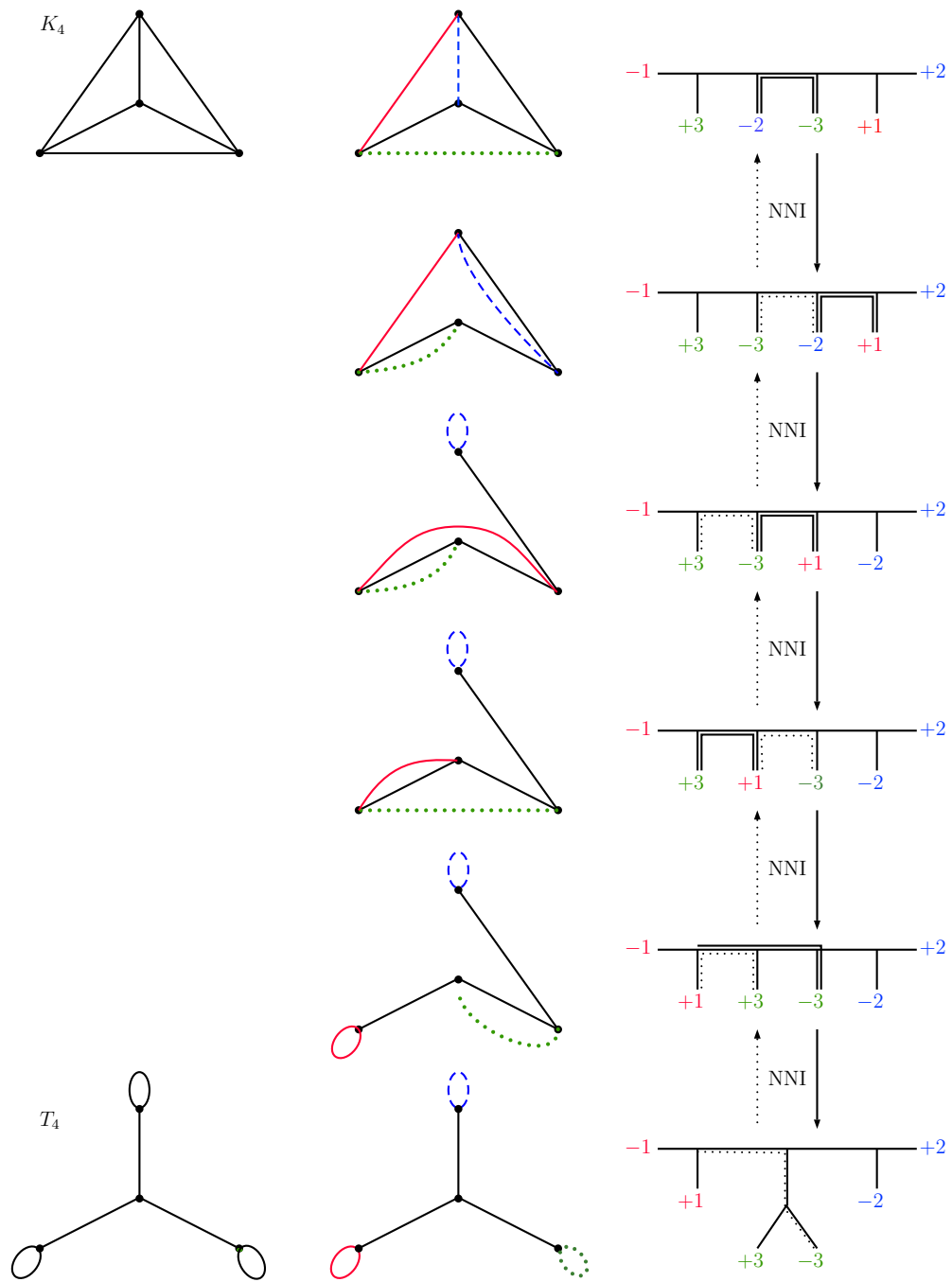}}
  \caption{A series of NNI moves from $K_4$ to the graph $T_4$.}
  \label{fig:sim}
\end{figure}

Figure~\ref{fig:sim} shows a series of NNI moves from the complete graph $K_4$ to a tree
with a loop added to each of its leaves.




\section{Weighted NNIs for $\{1,3\}$-graphs}\label{sec:wnni} 

To deal with weights on the edges of a $\{1,3\}$-graph, we now enhance the NNI move. 
This was achieved by a bijection defined by Wakabayashi~\cite[Prop.~6.3]{Wakabayashi2013}.

Let $G=(V,E)$ be a $\{1,3\}$-graph and $w$ be a weight function defined on the edges of~$G$. 
A \emph{weighted NNI} is a local move performed in $(G,w)$ on a trail~$W$ of length three
induced by a NNI move in $G$ on $W$.
The result of the move is the graph~$G'$ obtained from~$G$ by applying an NNI move on $W$, 
and the weight function $w'$ defined on the edges of $G'$ as follows.

Let $e$ be the central edge of $W$, that is, the pivot of the NNI move. 
Let $a$ and $b$ be the other edges in $W$, and $c$ and~$d$ be the remaining 
edges adjacent to $e$, as depicted in Figure~\ref{fig:weighted-nni}. 
Possibly $a$, $b$, $c$, and~$d$ are not pairwise distinct.  
The weight function $w'$ is such that $w'_f = w_f$ for every~$f \neq e$ and
$$ w'_e = w_e + \max\{w_a+w_c,w_b+w_d\} - \max\{w_b+w_c,w_a+w_d\}. $$ 
Note that, if $w$ is integer valued, then so is $w'$. 
Moreover, since pivots are always internal edges and an NNI move does not affect the 
partition of the edges into internal and external, a weighted NNI move may only change the 
weights of internal edges.

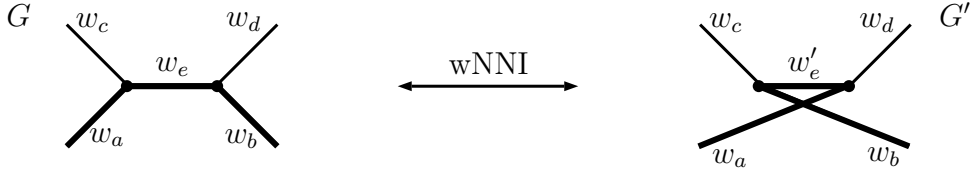
\begin{figure}[h]
  \centering
{
\begin{pspicture}(0,-1.0917871)(13.18,1.0917871)
\rput[bl](1.98,0.2041504){$w_e$}
\psline[linecolor=black, linewidth=0.04](1.6,0.08415039)(2.8,0.08415039)(3.6,0.8841504)
\psline[linecolor=black, linewidth=0.04](0.8,0.8841504)(1.6,0.08415039)(0.8,-0.7358496)
\rput[bl](2.9,-0.7558496){$w_b$}
\rput[bl](2.92,0.8041504){$w_d$}
\rput[bl](1.1,-0.7358496){$w_a$}
\rput[bl](0.92,0.78415036){$w_c$}
\psline[linecolor=black, linewidth=0.04](2.8,0.08415039)(3.6,-0.71584964)
\rput[bl](10.38,0.16415039){$w'_e$}
\psline[linecolor=black, linewidth=0.04](10.02,0.08415039)(11.22,0.08415039)(12.02,0.8841504)
\psline[linecolor=black, linewidth=0.04](9.22,0.8841504)(10.02,0.08415039)(12.0,-0.7358496)
\rput[bl](11.44,-0.9958496){$w_b$}
\rput[bl](11.34,0.78415036){$w_d$}
\rput[bl](9.38,-1.0158496){$w_a$}
\rput[bl](9.34,0.78415036){$w_c$}
\psline[linecolor=black, linewidth=0.08](11.22,0.08415039)(9.2,-0.7358496)
\psdots[linecolor=black, dotsize=0.16](1.6,0.06415039)
\psdots[linecolor=black, dotsize=0.16](2.8,0.06415039)
\psdots[linecolor=black, dotsize=0.16](10.0,0.06415039)
\psdots[linecolor=black, dotsize=0.16](11.2,0.06415039)
\psline[linecolor=black, linewidth=0.08](11.2,0.06415039)(10.0,0.06415039)
\psline[linecolor=black, linewidth=0.08](10.0,0.06415039)(12.0,-0.7358496)
\psline[linecolor=black, linewidth=0.08](1.6,0.06415039)(2.8,0.06415039)
\psline[linecolor=black, linewidth=0.08](2.8,0.06415039)(3.6,-0.7358496)
\psline[linecolor=black, linewidth=0.08](0.8,-0.7358496)(1.6,0.06415039)
\rput[bl](0.0,0.8641504){$G$}
\rput[bl](12.4,0.8641504){$G'$}
\psline[linecolor=black, linewidth=0.04, arrowsize=0.05291667cm 2.0,arrowlength=1.4,arrowinset=0.0]{<->}(5.2,0.06415039)(7.6,0.06415039)
\rput[bl](5.88,0.22415039){wNNI}
\end{pspicture}
}
  \caption{Edges and weights in a weighted nearest neighbor interchange.}
  \label{fig:weighted-nni}
\end{figure}

We think of a weighted NNI move as a function $\phi_{W}(G,w) = (G',w'),$
which extends the previously defined NNI move, $\gamma_W(G) = G'$. 
Note that $\phi_{W}(G',w') = (G,w)$, because $\gamma_W(G')=G$ and 
\begin{align}
  & w'_e + \max\{w'_b+w'_c,w'_a+w'_d\} - \max\{w'_a+w'_c,w'_b+w'_d\}  = \nonumber \\ 
  & w'_e + \max\{w_b+w_c,w_a+w_d\} - \max\{w_a+w_c,w_b+w_d\}  = \nonumber \\
  & w_e + \max\{w_a+w_c,w_b+w_d\} - \max\{w_b+w_c,w_a+w_d\} \nonumber \\
  & \phantom{w_e} + \max\{w_b+w_c,w_a+w_d\} - \max\{w_a+w_c,w_b+w_d\}  = w_e. \nonumber 
\end{align}

Wakabayashi~\cite[Prop. 6.3]{Wakabayashi2013} proved that $w \in t\calP_G$ if and only if 
$w' \in t\calP_{G'}$ for every integer $t \ge 0$, which implies Liu and Osserman's conjecture.  
We observe that this holds also for every real $t \ge 0$, and state it below as we will use 
it in the next section. 

\begin{lem}\label{lem:cone-involution}
  Let $G$ be a $\{1,3\}$-graph and $w$ be a weight function defined on the edges of~$G$.
  Let $W$ be a trail in $G$ of length three and suppose that $\phi_W(G,w)=(G',w')$. 
  Then $w \in t\calP_G$ if and only if $w' \in t\calP_{G'}$ for every $t \ge 0$. \hfill \qed
\end{lem}

Therefore there are distinct rational polytopes whose Ehrhart quasi-polynomials coincide for all real $t$.

As a weighted NNI changes only the weight of the pivot, which is always an
internal edge, the series of NNI moves from $G$ to $G'$ changes only the
weights of internal edges.  In fact, by Theorem~\ref{thm:connected}\eqref{thm:connected:b},
one can choose a spanning tree $T$ in~$G$ and a spanning tree~$T'$ in~$G'$
and require that the series of NNI moves uses as pivots only internal edges
of $T$ and $T'$.  As a consequence, the bijection from $\calP_G$ to
$\calP_{G'}$ keeps fixed the majority of the coordinates of the points.
Namely, it changes only coordinates that correspond to internal edges of the chosen
spanning trees.
Formally, the latter discussion provides the following as a corollary 
of Theorem~\ref{thm:connected}\eqref{thm:connected:b} and Lemma~\ref{lem:cone-involution}.
Let ${E = \{1,\ldots,m\}}$ and  $w: E \rightarrow \IR$ be a
weight function defined on the edges of~$G$.
If $X$ is a subset of $E$, then the \emph{restriction of $w$ to $X$} is the function
$w\vert^{}_X: X \rightarrow \IR$ such that $w\vert^{}_X(x) = w(x)$ for all $x \in X$.

\begin{cor}
  Let $G$ and $G'$ be connected $\{1,3\}$-graphs with the same number of vertices and on the same set $E$ of edges.
  Then there exists a bijection $\phi$ between $\calP_G$ and $\calP_{G'}$ such that, for $w'=\phi(w)$, 
  we have that $$w'\vert^{}_{X} = w\vert^{}_{X}$$ 
  where $E \setminus X$ is the set of internal edges of arbitrary spanning trees in $G$ and $G'$. 
\end{cor}


\section{Scissors congruence}\label{sec:decomp}



Haase and McAllister~\cite{haase2008} have raised Question~\ref{qst:haase} that 
can be thought of as an analogue of Hilbert's third problem (equidecomposability) 
for the unimodular group. 
We show that for polytopes associated to $\{1,3\}$-graphs a 
statement similar to Question~\ref{qst:haase} holds.


Let $G$ be a $\{1,3\}$-graph with edge set $E=\{1,\ldots,m\}$. 
Let $W$ be a trail in $G$ of length three and suppose that $\phi_W(G,w)=(G',w')$. 
As argued ahead, the function $\phi_W(G,\cdot) : \mathbb R^m \rarr \mathbb R^m$ is associated 
to one or two hyperplanes in $\mathbb R^m$ and two or four unimodular transformations. 
For short, let $\phi(w) = \phi_W(G,w)$ for every $w \in \mathbb R^m$. 
Let $a$, $b$, $c$, $d$, and~$e$ be as in Figure~\ref{fig:weighted-nni}, 
with $a$, $e$, and $b$ being the edges of $W$. 
Clearly $\phi$ is piecewise linear, namely, for $w \in \mathbb R^m$, 
$w'_f = w_f$ for every $f \neq e$ and 
\begin{subequations}
  \begin{align}
    w'_e & =  w_e  + w_b - w_d \, \quad \mbox{if $w_a+w_b \geq w_c+w_d$ and $w_a+w_d \geq w_b+w_c$}, \label{eq:uniA}\\ 
    w'_e & =  w_e  + w_a - w_c \, \quad \mbox{if $w_a+w_b \geq w_c+w_d$ and $w_a+w_d < w_b+w_c$}, \label{eq:uniB}\\ 
    w'_e & =  w_e  + w_c - w_a \, \quad \mbox{if $w_a+w_b < w_c+w_d$    and $w_a+w_d \geq w_b+w_c$},\label{eq:uniC}\\ 
    w'_e & =  w_e  + w_d - w_b \, \quad \mbox{if $w_a+w_b < w_c+w_d$    and $w_a+w_d < w_b+w_c$}. \label{eq:uniD}
  \end{align}
\end{subequations}
The hyperplanes associated to $\phi$ are $w_a+w_b-w_c-w_d = 0$ and $w_a-w_b-w_c+w_d = 0$, 
which are either the same hyperplane (if $a=b$ or $c=d$) or two orthogonal hyperplanes. 
Moreover, the matrix that gives the linear transformation in each case is unimodular.  Indeed, 
the matrix for case~\eqref{eq:uniA} is obtained from the identity matrix by substituting the row corresponding 
to the edge $e$ by the row $\chi^e + \chi^b - \chi^d$, 
the matrix for case~\eqref{eq:uniB} is obtained from the identity matrix by substituting the row corresponding 
to the edge $e$ by the row $\chi^e + \chi^a - \chi^c$, 
the matrix for case~\eqref{eq:uniC} is obtained from the identity matrix by substituting the row corresponding 
to the edge $e$ by the row $\chi^e + \chi^c - \chi^a$, and
the matrix for case~\eqref{eq:uniD} is obtained from the identity matrix by substituting the row corresponding 
to the edge $e$ by the row $\chi^e + \chi^d - \chi^b$.  
Thus the determinant of each such matrix is always~$1$.  
Therefore each of them is unimodular. 

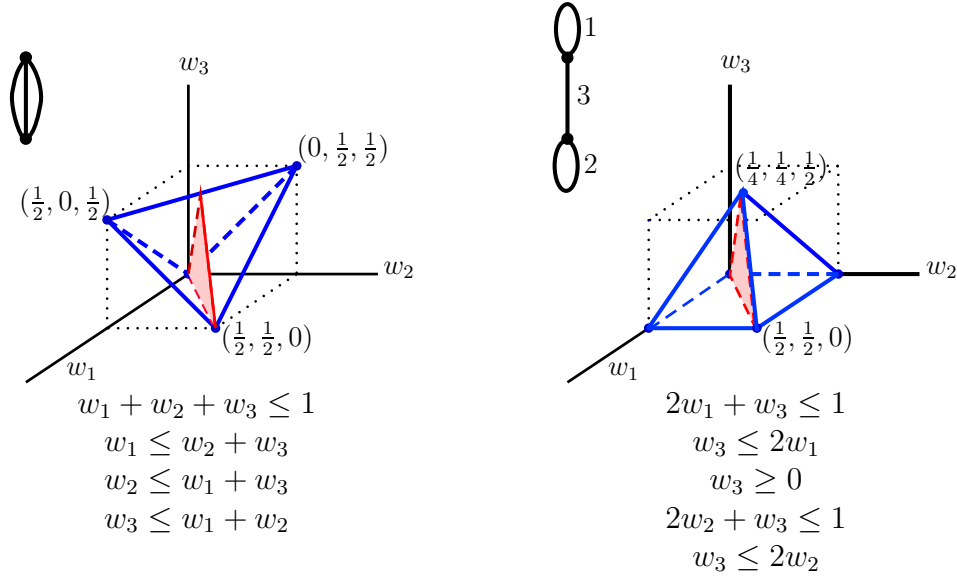
\begin{figure}[htb]
  \centering
  \scalebox{.9}{\psscalebox{1.0 1.0} 
{
\begin{pspicture}(0,-2.8183231)(13.963251,2.8183231)
\definecolor{colour4}{rgb}{1.0,0.0,0.2}
\definecolor{colour0}{rgb}{0.0,0.0,0.8}
\definecolor{colour1}{rgb}{0.0,0.2,1.0}
\definecolor{colour2}{rgb}{1.0,0.8,0.8}
\definecolor{colour3}{rgb}{0.8,0.0,0.0}
\psline[linecolor=black, linewidth=0.06](13.463251,-1.201677)(12.263251,-1.201677)
\psline[linecolor=black, linewidth=0.04](2.6632512,1.5983231)(2.6632512,-1.201677)(5.463251,-1.201677)
\psline[linecolor=black, linewidth=0.04](2.6632512,-1.201677)(0.2632512,-2.801677)
\psline[linecolor=black, linewidth=0.04, linestyle=dotted, dotsep=0.10583334cm](1.4632512,-2.001677)(3.0632513,-2.001677)(4.2632513,-1.201677)
\psline[linecolor=black, linewidth=0.04, linestyle=dotted, dotsep=0.10583334cm](4.2632513,-1.201677)(4.2632513,0.39832306)(2.6632512,0.39832306)(1.4632512,-0.40167695)(1.4632512,-2.001677)
\psline[linecolor=blue, linewidth=0.06](1.4632512,-0.40167695)(4.2632513,0.39832306)(3.0632513,-2.001677)(1.4632512,-0.40167695)
\psline[linecolor=blue, linewidth=0.06, linestyle=dashed, dash=0.17638889cm 0.10583334cm](1.4632512,-0.40167695)(2.6632512,-1.201677)(4.2632513,0.39832306)
\psline[linecolor=colour4, linewidth=0.06, linestyle=dashed, dash=0.17638889cm 0.10583334cm](2.6632512,-1.201677)(3.0632513,-2.001677)
\psdots[linecolor=colour0, dotsize=0.14](4.2632513,0.39832306)
\psdots[linecolor=colour0, dotsize=0.14](1.4632512,-0.40167695)
\psdots[linecolor=colour0, dotsize=0.14](3.0632513,-2.001677)
\psdots[linecolor=colour0, dotsize=0.14](2.6632512,-1.201677)
\rput[bl](4.2632513,0.39832306){$(0,\frac12,\frac12)$}
\rput[bl](3.1432512,-2.401677){$(\frac12,\frac12,0)$}
\rput[bl](0.16325119,-0.40167695){$(\frac12,0,\frac12)$}
\psdots[linecolor=black, dotsize=0.18](0.2632512,1.9983231)
\psdots[linecolor=black, dotsize=0.18](0.2632512,0.79832304)
\psline[linecolor=black, linewidth=0.06](0.2632512,1.9983231)(0.2632512,0.79832304)
\psarc[linecolor=black, linewidth=0.06, dimen=outer](0.2632512,1.9983231){0.06}{0.0}{270.0}
\psbezier[linecolor=black, linewidth=0.06](0.2632512,0.79832304)(0.39956892,0.85121626)(0.4632512,1.2183231)(0.48325118,1.3383231)(0.5032512,1.458323)(0.46515942,1.7927502)(0.2632512,1.9983231)
\psbezier[linecolor=black, linewidth=0.06](0.28024504,2.000461)(0.11141365,1.8227481)(0.09501528,1.6478094)(0.06301463,1.4910246)(0.031013966,1.33424)(0.1198138,0.9603481)(0.21500681,0.8204667)
\psdots[linecolor=black, dotsize=0.18](8.263251,1.9983231)
\psdots[linecolor=black, dotsize=0.18](8.263251,0.79832304)
\psline[linecolor=black, linewidth=0.06](8.263251,1.9983231)(8.263251,0.79832304)
\psellipse[linecolor=black, linewidth=0.06, dimen=outer](8.263251,2.418323)(0.2,0.4)
\psellipse[linecolor=black, linewidth=0.06, dimen=outer](8.243251,0.39832306)(0.2,0.4)
\psline[linecolor=black, linewidth=0.04](9.463251,-2.001677)(8.263251,-2.801677)
\psline[linecolor=black, linewidth=0.04, linestyle=dotted, dotsep=0.10583334cm](9.463251,-2.001677)(11.0632515,-2.001677)(12.263251,-1.201677)
\psline[linecolor=black, linewidth=0.04, linestyle=dotted, dotsep=0.10583334cm](12.263251,-1.201677)(12.263251,0.39832306)(10.663251,0.39832306)(9.463251,-0.40167695)(9.463251,-2.001677)
\psline[linecolor=blue, linewidth=0.06](10.843251,0.018323058)(12.263251,-1.201677)
\psdots[linecolor=colour0, dotsize=0.02625](9.463251,-0.40167695)
\psdots[linecolor=colour0, dotsize=0.14](9.463251,-2.001677)
\psdots[linecolor=colour0, dotsize=0.14](10.663251,-1.201677)
\rput[bl](10.723251,0.11832306){$(\frac14,\frac14,\frac12)$}
\rput[bl](11.123251,-2.401677){$(\frac12,\frac12,0)$}
\psdots[linecolor=colour0, dotsize=0.14](12.263251,-1.201677)
\psline[linecolor=black, linewidth=0.04, linestyle=dotted, dotsep=0.10583334cm](9.463251,-0.40167695)(11.0632515,-0.40167695)(12.263251,0.39832306)
\rput[bl](8.503251,2.3183231){1}
\rput[bl](8.503251,0.29832307){2}
\rput[bl](8.403252,1.318323){3}
\rput[bl](8.863251,-2.761677){$w_1$}
\rput[bl](13.5632515,-1.2616769){$w_2$}
\rput[bl](10.523252,1.778323){$w_3$}
\psline[linecolor=colour1, linewidth=0.06](9.463251,-2.001677)(9.463251,-2.001677)(11.0632515,-2.001677)(12.263251,-1.201677)
\psline[linecolor=colour1, linewidth=0.06, linestyle=dashed, dash=0.17638889cm 0.10583334cm](10.663251,-1.201677)(12.263251,-1.201677)
\psline[linecolor=black, linewidth=0.06](10.663251,1.5983231)(10.663251,-1.1416769)
\rput[bl](0.8632512,-2.761677){$w_1$}
\rput[bl](2.5232513,1.7583231){$w_3$}
\rput[bl](5.543251,-1.2616769){$w_2$}
\psline[linecolor=colour3, linewidth=0.04, linestyle=dashed, dash=0.17638889cm 0.10583334cm, fillstyle=solid,fillcolor=colour2](2.6632512,-1.201677)(2.8432512,-0.02167694)(3.0632513,-2.001677)
\psline[linecolor=red, linewidth=0.04](2.8432512,-0.041676942)(3.0632513,-1.9816769)
\psline[linecolor=colour1, linewidth=0.04, linestyle=dashed, dash=0.17638889cm 0.10583334cm](10.623251,-1.221677)(9.463251,-2.021677)(9.463251,-2.021677)
\psline[linecolor=red, linewidth=0.04, linestyle=dashed, dash=0.17638889cm 0.10583334cm, fillstyle=solid,fillcolor=colour2](11.0632515,-2.041677)(10.663251,-1.221677)(10.823251,-0.041676942)(11.023252,-1.941677)
\psdots[linecolor=colour0, dotsize=0.14](11.0632515,-2.001677)
\psdots[linecolor=colour1, dotsize=0.14](10.863251,-0.0016769409)
\psline[linecolor=colour1, linewidth=0.06](9.443252,-2.021677)(10.843251,0.018323058)(11.0632515,-2.041677)
\end{pspicture}
}}
\begin{tabular}{cp{3.8cm}cp{2mm}}
$w_1 + w_2 + w_3 \le 1$ & &  $2w_1 + w_3 \le 1$ & \\
$w_1 \le w_2 + w_3$ & & $w_3 \le 2w_1$ & \\
$w_2 \le w_1 + w_3$ & & $w_3 \geq 0$ & \\
$w_3 \le w_1 + w_2$ & & $2w_2 + w_3 \le 1$ & \\
                   & & $w_3 \le 2w_2$ & 
\end{tabular}
  \caption{The polytopes of the two cubic graphs on two vertices.  
  The shaded triangles are the intersection of the polytopes with the hyperplane $w_1 = w_2$.}
  \label{fig:involution}
\end{figure}

Figure~\ref{fig:involution} shows an example with the two 3-regular graphs on two vertices and their polytopes.
The two graphs differ by one NNI move.  
The function $\phi(w) = w'$ between the two polytopes in Figure~\ref{fig:involution}
is defined by $w'_1 = w_1$, $w'_2 = w_2$, and 
\begin{eqnarray*}
 w'_3 & = & w_3 + \max\{w_1+w_2,w_1+w_2\} - \max\{2w_1,2w_2\} \\
      & = & w_3 + w_1 + w_2 - 2\max\{w_1,w_2\} \\
      & = & \left\{\begin{array}{ll}
                     w_3 - w_1 + w_2 & \mbox{if $w_1 \geq w_2$} \\
                     w_3 + w_1 - w_2 & \mbox{if $w_1 \leq w_2$},
                   \end{array}\right.
\end{eqnarray*}
which is piecewise unimodular.
In this case, only one hyperplane (the one containing the shaded triangle inside the polytopes, $w_1 = w_2$) 
splits the polytopes into two, each part being unimodularly equivalent to one of the parts of the other polytope.  
The corresponding unimodular transformations are given by 
\begin{center}
$ U_{w_1 \leq w_2} = \left(\begin{array}{rrr}
                             \textcolor{white}{-}1 & 0 & \textcolor{white}{-}0 \\
                              0 & 1 & 0 \\
                              1 & -1 & 1
                            \end{array}\right) $
\hspace{1cm} $ U_{w_1 \geq w_2} = \left(\begin{array}{rrr}
                              1 & \textcolor{white}{-}0 & \textcolor{white}{-}0 \\
                              0 & 1 & 0 \\
                              -1 & 1 & 1
                            \end{array}\right). $
\end{center}

Note that the polytope on the right side of Figure~\ref{fig:involution} has one more vertex 
than the other one, so that the combinatorial types of these piecewise unimodularly equivalent 
polytopes may differ.  
Observe how the extra vertex $(\frac14,\frac14,\frac12)$ is ``formed'' when going from 
the polytope on the left to the one on the right, and how it ``disappears'' when going in 
the other direction.  Also, the edge between the origin and the vertex $(\frac12,\frac12,0)$ 
in the polytope on the left is not an edge in the polytope on the right. 

Theorem~\ref{thm:decomposition} extends this for graphs that differ by a series of NNI moves. 
It relates to Question~\ref{qst:haase}.  



\begin{proof}[Proof of Theorem~\ref{thm:decomposition}]
  The proof is constructive.  
  By Theorem~\ref{thm:connected}, there exists a finite
  sequence of NNI moves that transforms~$G$ to~$G'$, say
  $\label{thm:connected:a}
  G \x
  = G_0 \x
  \substack{{\scriptscriptstyle \mathrm{NNI}}\\\longleftrightarrow} \x
  \cdots \x
  \substack{{\scriptscriptstyle \mathrm{NNI}}\\\longleftrightarrow} \x G_k 
  = 
  G'. \x
  $
  We shall explain the procedure for the first two NNI moves assuming that both underlying 
  weighted NNI moves are associated to only one hyperplane and consequently two unimodular transformations. 
  The other cases and the rest of the NNI moves follow in an analogous but more cumbersome fashion. 

\begin{figure}
  \centering
  \scalebox{.98}{\input{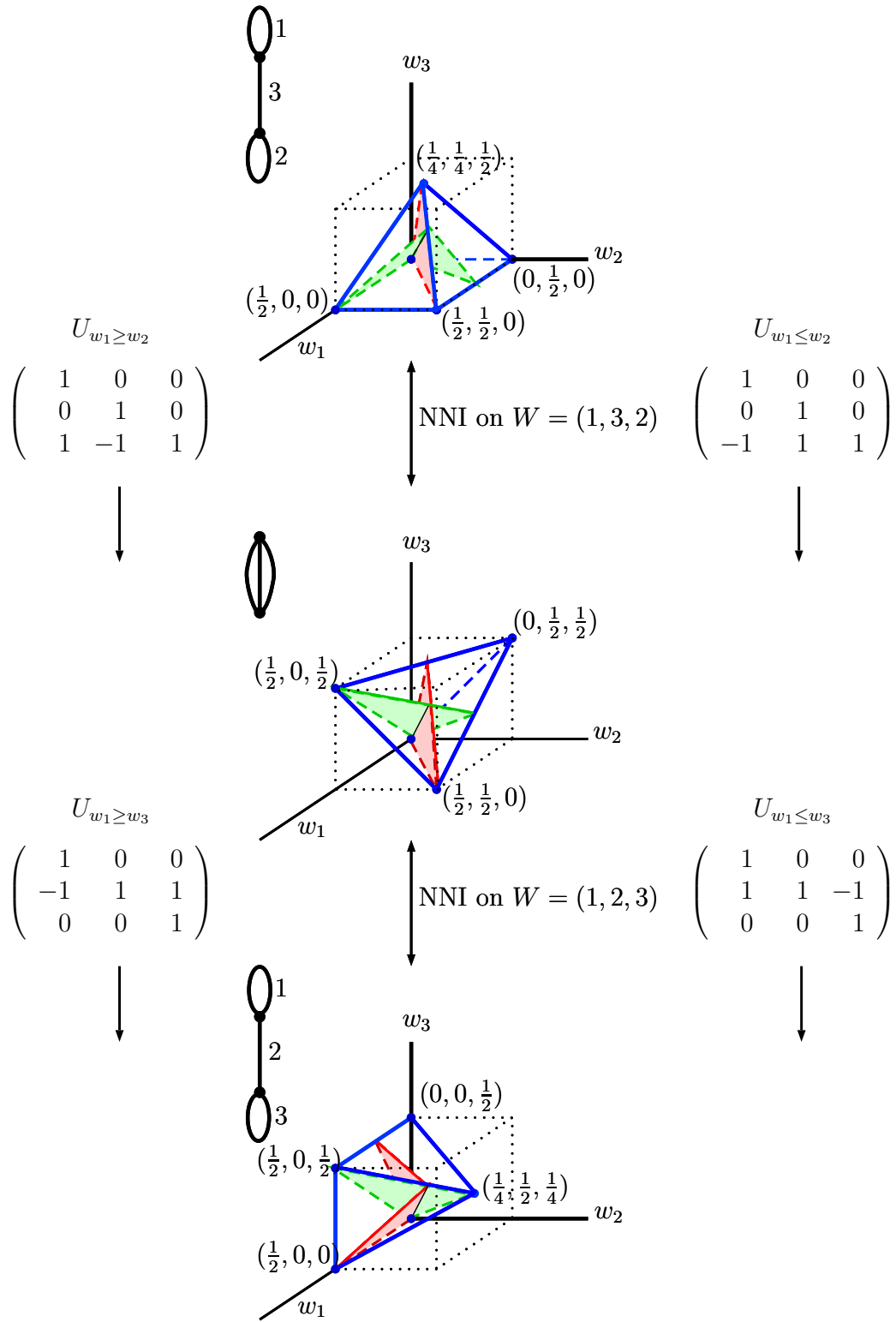}}
  \caption{Example for Theorem~\ref{thm:decomposition}.}
  \label{fig:unimodular}
\end{figure}

  Let $W_0$ be the trail used in the first NNI move, that goes from $G = G_0$ to $G_1$, 
  and let~$\phi_0(w) = \phi_{W_0}(G_0,w)$, for $w \in \mathbb R^E$, be the corresponding weighted NNI move.  
  Let $H_0$ be the hyperplane associated to $\phi_0$, and let $U_1$, $U_2$ be the two unimodular 
  transformations, one for each side of the hyperplane $H_0$.  The hyperplane $H_0$ dissects the 
  polytope~$\calP_G$ into two smaller polytopes, $\calQ^1_G$ and $\calQ^2_G$, so that 
  $\calP_G = \calQ_G^1 \cup \calQ_G^2$.  We now have that ${\calP_{G_1} = U_1(\calQ_G^1) \cup U_2(\calQ_G^2)}$.
  In words, we presented a dissection of $\calP_G$ into two smaller polytopes 
  and two affine unimodular transformations $U_1$ and $U_2$ that, if applied to the two smaller polytopes, 
  result in $\calP_{G_1}$.  Now we will proceed one more step, and present a dissection of $\calP_G$ 
  into four smaller polytopes, and four affine unimodular transformations that, if applied to the 
  four smaller polytopes, will result in $\calP_{G_2}$ (Figure~\ref{fig:unimodular}).

  Let $W_1$ be the trail used in the second NNI move, that goes from $G_1$ to $G_2$, 
  and let $\phi_1(w) = \phi_{W_1}(G_1,w)$, for $w \in \mathbb R^E$, be the corresponding weighted NNI move.  
  Let $H_1$ be the hyperplane associated to $\phi_1$, and let $T_1$, $T_2$ be the two unimodular 
  transformations, one for each side of the hyperplane $H_1$.  
  The hyperplane $H_1$ dissects the 
  polytope $\calP_{G_1}$ into two smaller polytopes, $\calQ^1_{G_1}$ and $\calQ^2_{G_1}$, 
  so that $\calP_{G_2} = T_1(\calQ_{G_1}^1) \cup T_2(\calQ_{G_1}^2)$. 

  Now, note that $H_1$ dissects $U_1(\calQ_G^1)$ and $U_2(\calQ_G^2)$ each into two smaller polytopes, 
  obtaining the dissection
  $$\calP_{G_1} = (\calQ_{G_1}^1 \cap U_1(\calQ_G^1))  \cup  (\calQ_{G_1}^1 \cap U_2(\calQ_G^2)) 
                  \cup (\calQ_{G_1}^2 \cap U_1(\calQ_G^1)) \cup (\calQ_{G_1}^2 \cap U_2(\calQ_G^2)).$$
  The latter naturally induces the following dissection:
  $$\calP_G = U_1^{-1}(\calQ_{G_1}^1 \cap U_1(\calQ_G^1)) \cup U_2^{-1}(\calQ_{G_1}^1 \cap U_2(\calQ_G^2)) 
             \cup U_1^{-1}(\calQ_{G_1}^2 \cap U_1(\calQ_G^1)) \cup U_2^{-1}(\calQ_{G_1}^2 \cap U_2(\calQ_G^2)).$$
  So, if we let \vspace{-3mm}
  \begin{eqnarray*}
    \calP^{11} & = & U_1^{-1}(\calQ_{G_1}^1 \cap U_1(\calQ_G^1)) \ = \ U_1^{-1}(\calQ_{G_1}^1) \cap \calQ_G^1, \\ 
    \calP^{12} & = & U_2^{-1}(\calQ_{G_1}^1 \cap U_2(\calQ_G^2)) \ = \ U_2^{-1}(\calQ_{G_1}^1) \cap \calQ_G^2, \\
    \calP^{21} & = & U_1^{-1}(\calQ_{G_1}^2 \cap U_1(\calQ_G^1)) \ = \ U_1^{-1}(\calQ_{G_1}^2) \cap \calQ_G^1, \\
    \calP^{22} & = & U_2^{-1}(\calQ_{G_1}^2 \cap U_2(\calQ_G^2)) \ = \ U_2^{-1}(\calQ_{G_1}^2) \cap \calQ_G^2,
  \end{eqnarray*}
  then 
  \begin{eqnarray*}
  \calP_G & = & \calP^{11} \cup \calP^{12} \cup \calP^{21} \cup \calP^{22}, \mbox{\ and } \\
  \calP_{G_2} & = & T_1(\calQ_{G_1}^1) \cup T_2(\calQ_{G_1}^2)\\
          & = & \left(T_1(\calQ_{G_1}^1 \cap U_1(\calQ_G^1)) \cup T_1(\calQ_{G_1}^1 \cap U_2(\calQ_G^2))\right) \\
          &   & \cup \left(T_2(\calQ_{G_1}^2 \cap U_1(\calQ_G^1)) \cup T_2(\calQ_{G_1}^2 \cap U_2(\calQ_G^2))\right)\\
          & = & \left(T_1U_1(\calP^{11}) \cup T_1U_2(\calP^{12})\right) \cup \left(T_2U_1(\calP^{21}) \cup T_2U_2(\calP^{22})\right).
  \end{eqnarray*}
  This completes the proof for the two first weighted NNI moves assuming that both are
  associated to only one hyperplane and consequently two unimodular transformations. 

  For the remaining cases, whenever a weighted NNI move is associated to two hyperplanes 
  (and consequently four unimodular transformations), the polytopes would be dissected 
  into up to four smaller polytopes, but the process would be essentially the same. 
\end{proof}


\section{Reflexivity}
\label{sec:reflexivity}


An equivalent way to define a reflexive polytope is to require that it is integral 
and has the hyperplane description $\{x \in \IR^d \ | \ Ax \le \mathbbm{1} \}$ for some
integral matrix~$A$.  Another equivalent definition of reflexivity is to say that 
a polytope $\calP$ is reflexive if and only if the origin is in $\calP^\circ$ and 
$(t+1)\calP^\circ \cap \ZZ^d \ = \ t\calP \cap \ZZ^d \mbox{\ \ for all $t \in \ZZ_{\ge 0}$},$
where $\calP^\circ$ is the interior of~$\calP$.

\begin{proof}[Proof of Theorem~\ref{thm:reflexive}]
The polytope $4\calP_G{-}\mathbbm{1}$ consists of the vectors $w \in \IR^m$
satisfying
\begin{eqnarray*}
   \ph{+}w_a + w_b + w_c & \leq & 1 \\
   \ph{+}w_a - w_b - w_c & \leq & 1 \\
   -w_a + w_b - w_c & \leq & 1 \\
   -w_a - w_b + w_c & \leq & 1,
\end{eqnarray*}
for each degree three vertex $v$ and edges $a$,~$b$, and $c$ incident to
$v$,
for a total of $4n_3$ inequalities, where $n_3$ is the number of degree
three vertices in $G$.
From Liu and Osserman~\cite[Prop.~3.5]{LiuO2006}, the polytope $4\calP_G$
is integral, and thus so is
$4\calP_G{-}\mathbbm{1}$.
\end{proof}

The reflexivity of the integral polytope $4\calP_G{-}\mathbbm{1}$ has some
interesting consequences, as follows.
Let  $m$ denote the number of edges of $G$.  Since $4\calP_G$ is an
$m$-dimensional integral polytope, we have that for $t \equiv 0 \mod 4$:
\[
L_{\calP_G}(t)  = \binom{t+m}{m} + h_1^* \binom{t+m-1}{m} + \cdots +
h_{d-1}^* \binom{t+1}{m} + h_d^*\binom{t}{d},
\]
where $h_1^*, \ldots, h_{d-1}^*, h_d^*$ are the coefficients of the numerator 
of the rational function given by the Ehrhart series of $4\calP_G$~\cite[Lemma~3.14]{niceperson}.

The polytopes $4\calP_G$ and $4\calP_G-\mathbbm{1}$ have the same Ehrhart
polynomial.
Hence it follows from the reflexivity of  $4\calP_G{-}\mathbbm{1}$ and from
Hibi's palindromic
theorem~\cite[Thm.~4.6]{niceperson}
that $h_k^* = h_{m-k}^*$ for all $0 \leq k \leq m/2$.  Therefore it follows
that
to describe $L_{4\calP_G}(t)$,  we only need to compute half of its
coefficients, thus lowering the computational complexity required to
compute $L_{4\calP_G}(t)$.

\section{Concluding remarks and open problems}\label{sec:remarks}

In the process of proving the main theorem, a great deal of structure of the
polytopes~$\calP_G$, and their corresponding cones has been revealed.
The polytopes $\calP_G$ and their corresponding cones may be related to the 
well-known metric polytopes and metric cones~\cite{Deza:2009}.
In particular, in the case that $G$ is a planar graph, the polytope $\calP_G$ 
can be thought of as a restricted type of ``metric polytope''~\cite{Laurent1996},
associated to the dual graph $G^*$.  
It would be interesting to pursue connections to metric polytopes. 

Wakabayashi~\cite{Wakabayashi2013} gave a formula for two of the four constituent polynomials of
the Ehrhart quasi-polynomial for $\calP_G$ and for the volume of
$\calP_G$ when $G$ is a connected cubic graph on~$n$ vertices.
It is rather surprising that the Verlinde formula makes
an appearance here, namely, Wakabayashi showed that, for odd~$t$,
$$ L_{\calP_G}(t) \ = \ \frac{(t+2)^{n/2}}{2^{n+1}} \sum_{j=1}^{t+1} \frac1{\sin^n\big(\frac{\pi j}{t+2}\big)} 
\, \, \mbox{  and  } \, \, 
\vol(\calP_G) \ = \ \frac{|B_n|}{2\,n!}, $$
where $B_n$ is the $n$-th Bernoulli number.
Zagier~\cite{Zagier1996} showed that the above Verlinde formula 
is the polynomial
$$ L_{\calP_G}(t) \ = \ \frac{(t+2)^{n/2}}{2^{n+1}} \sum_{k=0}^{n/2} \frac{(-1)^{k-1}2^{2k}B_{2k}}{(2k)!}c_k(t+2)^{2k}, $$
where $c_k$ is the coefficient of $x^{-2k}$ in the Laurent expansion of $\sin^{-n}x$ at $x = 0$.
It would be interesting to determine if a similar formula
holds for all connected $\{1,3\}$-graphs.
Even when $G$ is a
$\{1,3\}$-tree, it is completely open, and quite interesting, to determine
the Ehrhart quasi-polynomial of $\calP_G$.
We determined the following quasi-polynomials  for some small trees with
LattE~\cite{latte}.


\begin{center}
\noindent
\begin{tabular}{|>{\centering\arraybackslash}m{0.8cm}|l|} \hline
  $G$ & Ehrhart polynomial $\calP_G$\\ \hline
  \includegraphics[height=0.8cm]{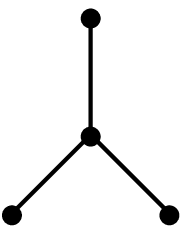}&%
  $\frac{1}{24} t^3 + \frac{1}{4} t^2 + \,
    \left\{\begin{array}{ll}
    \frac{5}{6} t + 1,   & \mbox{if $t$ is even} \\[1mm] 
    \frac{11}{24} t + \frac{1}{4}, & \mbox{if $t$ is odd}
    \end{array}\right.%
  $
    \\ \hline
    \includegraphics[height=1cm]{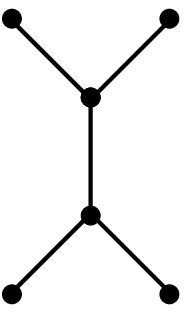} &%
    ${\small
    \frac{1}{240} t^5 + \frac{1}{24} t^4 + \,
    \left\{\begin{array}{ll}
    \frac{5}{24} t^3 + \frac{7}{12} t^2 +
    \frac{11}{10}t + 1,
    & \mbox{if $t$ is even} \\[1mm] 
    \frac{1}{6} t^3 + \frac{1}{3} t^2 +
    \frac{79}{240}t + \frac{1}{8},
    & \mbox{if $t$ is odd}
    \end{array}\right.%
    }
    $
    \\ \hline
    \includegraphics[height=1cm]{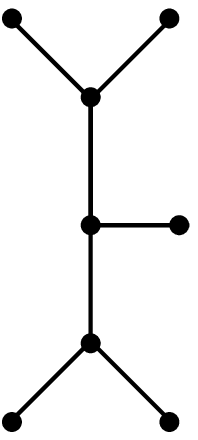} &%
    $
    {\tiny
    \frac{17}{40320}t^7 + \frac{17}{2880}t^6 + \,
    \left\{\begin{array}{ll}
    \frac{59}{1440} t^5 + \frac{25}{144} t^4 +
    \frac{179}{360} t^3 + \frac{173}{180} t^2 +
    \frac{93}{70}t + 1,   & \mbox{if $t$ is even} \\[1mm] 
    \frac{103}{2880} t^5 + \frac{35}{288} t^4 +
    \frac{1439}{5760} t^3 + \frac{893}{2880} t^2 +
    \frac{791}{3360}t + \frac{1}{16}, & \mbox{if $t$ is odd}
    \end{array}\right.
    }
    $
    \\ \hline
    \includegraphics[height=1cm]{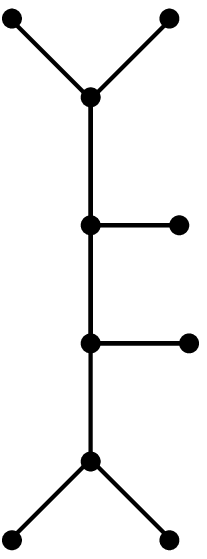} &%
    $
    {\tiny
    \frac{31}{725760}t^9 + \frac{31}{40320}t^8 + \,
    \left\{\begin{array}{ll}
    \frac{829}{120960}t^7 + \frac{37}{960}t^6 +
    \frac{653}{4320} t^5 + \frac{103}{240} t^4 +
    \frac{20413}{22680} t^3 + \frac{1723}{1260} t^2 +
    \frac{193}{126}t + 1,  & \mbox{if $t$ is even} \\[1mm] 
    \frac{43}{6912}t^7 + \frac{19}{640}t^6 +
    \frac{3181}{34560} t^5 + \frac{123}{640} t^4 +
    \frac{39205}{145152} t^3 + \frac{9923}{40320} t^2 +
    \frac{379}{2880}t + \frac{1}{32}, &
    \mbox{if $t$ is odd}
    \end{array}\right.
    }
    $
    \\ \hline
\end{tabular}
\end{center}

Further generalizations to more general graphs
and connections with other topics as manifold invariants would also be of interest.

\section{Acknowledgements}\label{sec:acknowledgements}

We thank Tiago Royer and Fabr\'icio Caluza Machado for valuable comments.

\bibliographystyle{plain}
\bibliography{CubicQGraphs}

\begin{thebibliography}{10}

\bibitem{latte}
V.~Baldoni, N.~Berline, J.A.~De Loera, B.~Dutra, M.~K\"oppe, S.~Moreinis,
  G.~Pinto, M.~Vergne, and J.~Wu.
\newblock Software and user's guide for {L}att{E} integrale, October 2014.
\newblock \texttt{https://www.math.ucdavis.ed/\~\null latte}.

\bibitem{niceperson}
M.~Beck and S.~Robins.
\newblock {\em Computing the Continuous Discretely}.
\newblock Springer, second edition, 2009.

\bibitem{BondyM2008}
J.A. Bondy and U.S.R. Murty.
\newblock {\em Graph Theory}.
\newblock Springer, 2008.

\bibitem{CulikW1982}
K.~Culik and D.~Wood.
\newblock A note on some tree similarity measures.
\newblock {\em Information Processing Letters}, 15(1):39--42, 1982.

\bibitem{Deza:2009}
M.M. Deza and M.~Laurent.
\newblock {\em Geometry of Cuts and Metrics}.
\newblock Springer Publishing Company, Incorporated, 1st edition, 2009.

\bibitem{ehrhartbook}
E.~Ehrhart.
\newblock {\em Polyn\^omes arithm\'etiques et m\'ethode des poly\`edres en
  combinatoire}.
\newblock Birkh\"auser Verlag, Basel, 1977.
\newblock International Series of Numerical Mathematics, Vol. 35.

\bibitem{haase2008}
C.~Haase and Tyrrell~B. McAllister.
\newblock Quasi-period collapse and {$\textup{GL}_n(\mathbb{Z})$}-scissors
  congruence in rational polytopes.
\newblock {\em Contemporary Mathematics}, 452(2008):115--122, 2008.

\bibitem{Laurent1996}
M.~Laurent.
\newblock Graphic vertices of the metric polytope.
\newblock {\em Discrete Mathematics}, 151:131--153, 1996.

\bibitem{Linke2011}
E.~Linke.
\newblock Rational {E}hrhart quasi-polynomials.
\newblock {\em J. Combin. Theory Ser. A}, 118(7):1966--1978, 2011.

\bibitem{LiuO2006}
F.~Liu and B.~Osserman.
\newblock Mochizuki’s indigenous bundles and {E}hrhart polynomials.
\newblock {\em Journal of Algebraic Combinatorics}, 23:125–136, 2006.

\bibitem{Mochizuki1996}
S.~Mochizuki.
\newblock {\em A Theory of Ordinary $p$-Adic Curves}, volume~32.
\newblock Publ.\ RIMS.\ Kyoto University, 1996.
\newblock 957--1151.

\bibitem{Robinson1971}
D.F. Robinson.
\newblock Comparison of labeled trees with valency three.
\newblock {\em Journal of Combinatorial Theory Ser. B}, 11:105--119, 1971.

\bibitem{Royer2017}
T.~Royer.
\newblock Semi-reflexive polytopes.
\newblock Available in the ArXiv \texttt{https://arxiv.org/abs/1712.04381},
  2017.

\bibitem{Tsukui1996}
Y.~Tsukui.
\newblock Transformations of cubic graphs.
\newblock {\em Journal of the Franklin Institute}, 333(B)(4):565--575, 1996.

\bibitem{Wakabayashi2013}
Y.~Wakabayashi.
\newblock Spin networks , {E}hrhart quasi-polynomials, and combinatorics of
  dormant indigenous bundles.
\newblock RIMS Preprint 1786, August 2013.

\bibitem{Zagier1996}
D.~Zagier.
\newblock Elementary aspects of the {V}erlinde formula and of the
  {H}arder-{N}arasimhan-{A}tiyah-{B}ott formula.
\newblock In {\em Israel Math. Conf. Proc.}, volume~9, pages 445--462, 1996.

\end{thebibliography}

\end{document}